\documentclass[12pt]{amsart}
\usepackage{amsmath,amssymb,amsfonts,amsthm}
\usepackage[dvips]{graphicx,epsfig,color}
\usepackage{enumerate}
\usepackage[mathscr]{eucal}
\usepackage{float}
\numberwithin{equation}{section}
\newtheorem{theorem}{Theorem}[section]
\newtheorem{definition}{Definition}[section]
\newtheorem{lemma}[theorem]{Lemma}
\newtheorem{proposition}[theorem]{Proposition}
\newtheorem{corollary}[theorem]{Corollary}
\newtheorem{conjecture}[theorem]{Conjecture}

\numberwithin{equation}{section}

 \makeatletter
      \def\@setcopyright{}
      \def\serieslogo@{}
      \makeatother

\begin{document}

\title{Stability in Respect of Chromatic Completion of Graphs} 
\author{E.G. Mphako-Banda$^*$ and J. Kok$^{\dag}$} 
\address{$^*$School of Mathematical Sciences\\
University of Witwatersrand\\
Johannesburg, South Africa\\
$^{\dag}$Centre for Studies in Discrete Mathematics\\
Vidya Academy of Science $\&$ Technology\\
Thrissur, India}
\email{$^*$~eunice.mphako-banda@wits.ac.za\\
$^{\dag}$~kokkiek2@tshwane.gov.za}



\begin{abstract}
In an improper colouring an edge $uv$ for which, $c(u)=c(v)$ is called a \emph{bad edge}.  The notion of the \emph{chromatic completion number} of a graph $G$ denoted by $\zeta(G),$ is the maximum number of edges over all chromatic colourings that can be added to $G$ without adding a bad edge. We introduce stability of a graph in respect of chromatic completion. We prove that the set of chromatic completion edges denoted by  $E_\chi(G),$ which corresponds to $\zeta(G)$ is unique if and only if $G$ is stable in respect of chromatic completion. Thereafter, chromatic completion and stability is discussed in respect of Johan colouring. The difficulty of studying chromatic completion with regards to graph operations is shown by presenting results for two elementary graph operations.
\end{abstract}
\maketitle
\section{Introduction}
For general notation and concepts in graphs see \cite{1,2,13}.  Unless stated otherwise, all graphs will be finite and simple, connected graphs with at least one edge. The set of vertices and the set of edges of a graph $G$ are denoted by, $V(G),$ $E(G),$ respectively. The number of vertices is denoted by, $\nu(G)$ and the number of edges of $G$ is denoted by, $\varepsilon(G).$ The degree of a vertex $v \in V(G)$ is denoted $d_G(v)$ or when the context is clear, simply as $d(v).$ The minimum and maximum degree $\delta(G)$ and $\Delta(G)$ respectively, have the conventional meaning. When the context is clear we shall abbreviate to $\delta$ and $\Delta,$ respectively. Recall that the set of vertices adjacent to a vertex $u\in V(G)$ is the open neighbourhood $N(u)$ of $u$ and the closed neighbourhood of $u$ is, $N[u] = N(u)\cup \{u\}.$

For a set of distinct colours $\mathcal{C}= \{c_1,c_2,c_3,\dots,c_\ell\}$ a vertex colouring of a graph $G$ is an assignment $\varphi:V(G) \mapsto \mathcal{C}.$ A vertex colouring is said to be a \textit{proper vertex colouring} of a graph $G$ if no two distinct adjacent vertices have the same colour. The cardinality of a minimum set of distinct colours in a proper vertex colouring of $G$ is called the \textit{chromatic number} of $G$ and is denoted by, $\chi(G).$ We call such a colouring a $\chi$-colouring or a \textit{chromatic colouring} of $G.$ A chromatic colouring of $G$ is denoted by $\varphi_\chi(G).$ Generally a graph $G$ of order $n$ is $k$-colourable for $\chi(G) \leq k.$ Unless mentioned otherwise, a set of colours will mean a set of distinct colours.

Generally the set, $c(V(G)) \subseteq \mathcal{C}.$ A set $\{c_i \in \mathcal{C}: c(v)=c_i$ for at least one $v\in V(G)\}$ is called a colour class of the colouring of $G.$ If $\mathcal{C}$ is the chromatic set it can be agreed that $c(G)$ means set $c(V(G))$ hence, $c(G) \Rightarrow \mathcal{C}$ and $|c(G)| = |\mathcal{C}|.$ For the set of vertices $X\subseteq V(G),$ the induced subgraph induced by $X$ is denoted by, $\langle X\rangle.$ The colouring of $\langle X\rangle$ permitted by $\varphi:V(G) \mapsto \mathcal{C}$ is denoted by, $c(\langle X\rangle).$ The number of times a colour $c_i$ is allocated to vertices of a graph $G$ is denoted by $\theta_G(c_i)$ or if the context is clear simply, $\theta(c_i).$

Index labeling the elements of a graph such as the vertices say, $v_1,v_2,v_3,\\ \dots,v_n$ or written as, $v_i$, $i = 1,2,3,\dots,n,$ is called minimum parameter indexing. Similarly, a \textit{minimum parameter colouring} of a graph $G$ is a proper colouring of $G$ which consists of the colours $c_i;\ 1\le i\le \ell.$ All graphs will have index labeled vertices. Only if the context is clear will general references to vertices such as $u \in V(G)$ or $v,w \in V(G)$ be used.

This paper is organised as follows. Section~\ref{s2} recalls important results of the \emph{chromatic completion number} of a graph $G$ from \cite{10}. The concept of stability in respect of chromatic completion of graphs is introduced. A uniqueness theorem in respect of the set of chromatic completion edges denoted by, $E_\chi(G)$ is also presented. Section~\ref{s3} addresses chromatic completion and stability of a graph in respect of Johan colourings. Section~\ref{s4} concludes this paper by discussing results for two elementary graph operations. 
 
\section{Stability}
\label{s2}
In an improper colouring an edge $uv$ for which, $c(u)=c(v)$ is called a \emph{bad edge}. See \cite{9} for an introduction to $k$-defect colouring and corresponding polynomials. For a colour set $\mathcal{C},$ $|\mathcal{C}| \geq \chi (G)$ a graph $G$ can always be coloured properly hence, such that no bad edge results. Also, for a set of colours $\mathcal{C},$ $|\mathcal{C}| = \chi (G) \geq 2$ a graph $G$ of order $n$ with corresponding chromatic polynomial $\mathcal{P}_G(\lambda, n),$ can always be coloured properly in $\mathcal{P}_G(\chi, n)$ distinct ways.  The notion of the \emph{chromatic completion number} of a graph $G$ denoted by, $\zeta(G)$ is the maximum number of edges over all chromatic colourings that can be added to $G$ without adding a bad edge \cite{10}. The resultant graph $G_\zeta$ is called a \emph{chromatic completion graph} of $G.$ The additional edges are called \emph{chromatic completion edges}. It is trivially true that $G\subseteq G_\zeta.$ Clearly for a complete graph $K_n,$ $\zeta(K_n)=0.$ In fact for any complete $\ell$-partite graph $H=K_{n_1,n_2,n_3,\dots,n_\ell},$ $\zeta(H)=0.$ Hereafter, all graphs will not be $\ell$-partite complete. For graphs $G,$ $H$ both of order $n$ with $\varepsilon(G)\geq \varepsilon(H)$ no relation between $\zeta(G)$ and $\zeta(H)$ could be found. We state without proof the following six important results from \cite{10}, which form a basis of this paper.
\begin{theorem}
A graph $G$ of order $n$ is not complete, if and only if $G_\zeta$ is not complete.
\label{thm2.1}
\end{theorem}
\begin{lemma}
\label{lem2.2}
For a chromatic colouring $\varphi:V(G)\mapsto \mathcal{C}$ a pseudo completion graph, $H(\varphi)= K_{n_1,n_2,n_3,\dots,n_\chi}$ exists such that, $$\varepsilon(H(\varphi))-\varepsilon(G) =\sum\limits_{i=1}^{\chi-1}\theta_G(c_i)\theta_G(c_j)_{(j=i+1,i+2,i+3,\dots,\chi)}-\varepsilon(G) \leq \zeta(G).$$
\end{lemma}
A main result in the form of a corollary is a direct consequence of Lemma 2.2.
\begin{corollary}
\label{col2.3}
Let  $G$ be a graph. Then 
\begin{eqnarray*}
\zeta(G) & =& max(\varepsilon(H(\varphi)) -\varepsilon(G) \text{ over all} \ \varphi:V(G)\mapsto \mathcal{C}.
\end{eqnarray*}
\end{corollary}
\begin{theorem}
\label{thm2.4}
Let  $G$ be a graph. Then $\zeta(G)\leq \varepsilon(\overline{G}).$
\end{theorem}
If for all chromatic colourings of $G$ we have that, $\theta(c_i)\geq 2$ for some $c_i$ then, $\zeta(G) < \varepsilon(\overline{G})$. Hence, equality holds if and only if a graph $G$ of order $n$ is complete.
\begin{theorem}[(Lucky's~Theorem)]
\label{thm2.5}
For a positive integer $n \geq 2$ and $2\leq p \leq n$ let integers, $1\leq a_1,a_2,a_3,\dots,a_{p-r}, a'_1,a'_2,a'_3,\dots,a'_r \leq n-1$ be such that $n=\sum\limits_{i=1}^{p-r}a_i + \sum\limits_{j=1}^{r}a'_j$ then, the $\ell$-completion sum-product $\mathcal{L} = max\{\sum\limits_{i=1}^{p-r-1}\prod\limits_{k=i+1}^{p-r}a_ia_k + \sum\limits_{i=1}^{p-r}\prod\limits_{j=1}^{r}a_ia'_j + \sum\limits_{j=1}^{r-1}\prod\limits_{k=j+1}^{r}a'_ja'_k\}$ over all possible, $n=\sum\limits_{i=1}^{p-r}a_i + \sum\limits_{j=1}^{r}a'_j.$
\end{theorem}
From Theorem 2.5 the next lemma followed which prescribes a particular colouring convention.
\begin{lemma}
\label{lem2.6}
If a subset of $m$ vertices say, $X \subseteq V(G)$ can be chromatically coloured by $t$ distinct colours then allocate colours as follows:
\begin{enumerate}[(a)]
\item For $t$ vertex subsets each of cardinality $s= \lfloor \frac{m}{t}\rfloor$ allocate a distinct colour followed by:
\item  Colour one additional vertex (from the $r\geq 0$ which are uncoloured), each in a distinct colour,
\end{enumerate}
if the graph structure permits such colour allocation.
This chromatic colouring permits the maximum number of chromatic completion edges between the vertices in $X$ amongst all possible chromatic colourings of $X$.
\end{lemma}
It is known that for a graph which does not permit a colour allocation as prescribed in Lemma~\ref{lem2.6}, an optimal near-completion $\ell$-partition of the vertex set exists which yields the maximum chromatic completion edges \cite{7}. Note that the colouring in accordance with Lemma~\ref{lem2.6} is essentially a special case of an optimal near-completion $\ell$-partition of the vertex set $V(G).$ Henceforth, a chromatic colouring in accordance with either Lemma~\ref{lem2.6} or an optimal near-completion $\ell$-partition will be called a \emph{Lucky colouring} denoted by, $\varphi_{\mathcal{L}}(G).$ If all possible Lucky colourings of a graph $G,$ yield identical vertex partitions then graph $G$ is said to be, \emph{stable in respect of chromatic completion}. Such graph is denoted to be, $SCC.$ It means that if $\chi(G)\geq 2,$ the different Lucky colourings only effect pairwise interchange of colour classes. Such different colourings are said to be \emph{congruent} and is denoted by, $\varphi_{\mathcal{L}}(G)_1 \cong \varphi_{\mathcal{L}}(G)_2.$ For all graphs for which $\zeta(G)=0$ it follows that $E_\chi(G) =\emptyset$ and therefore, inherently unique. All such graphs are inherently $SCC.$ Unless mentioned otherwise, graphs for which $\zeta(G)>0$ will be considered hereafter.
\begin{theorem}
\label{thm2.7}
For a graph $G$, $\zeta(G) >0$ is $SCC$ if and only if the chromatic completion edge set $E_\chi(G)$ is unique.
\end{theorem}
\begin{proof}
Let the vertices of a graph $G$ of order $n\geq 1$ be, $v_i$, $i=1,2,3,\dots, n.$ Assume that the chromatic completion edge set $E_\chi(G)$ is unique. Then all possible Lucky colourings of $G$ yield identical vertex partitions with only possible interchange of colour classes. By definition $G$ is $SCC.$

Converse: Assume that $G$ is $SCC$ the vertex partitions over all possible Lucky colourings are identical. Since, $\zeta(G)>0$ at least one edge $v_iv_j\in E_\chi(G)$ exists. It means that, if for any given Lucky colouring the edge $v_iv_j \in E_\chi(G)$ then, $v_iv_j \notin E(G)$ and $c(v_i)\neq c(v_j).$ It also means that in any other Lucky colouring, $c(v_j)\neq c(v_j)$ hence, $v_iv_j \in E_\chi(G)$ over all Lucky colourings. Therefore $E_\chi(G)$ is unique.
\end{proof}
\begin{corollary}
\label{col2.8}
For a graph $G,$ $\zeta(G)>0,$ the chromatic completion edge set $E_\chi(G),$ is unique if $G$ is $2$-colourable.
\end{corollary}
\begin{proof}
Consider any $2$-colourable graph $G,$ $\zeta(G) > 0.$ The vertex set can be partitioned in two unique subsets. Only two Lucky colourings are possible i.e. interchanging colours $c_1$ and $c_2$ hence, $G$ is $SCC.$ By Theorem~\ref{thm2.7} the result follows.
\end{proof}
\begin{corollary}
\label{col2.9}
A graph $G,$ $\chi(G)\geq 3$ which has a pendant vertex is not $SCC.$
\end{corollary}
\begin{proof}
Let $u$ be a pendant vertex and assume without loss of generality that, $c(u)=c_1.$ Also assume $u$ is adjacent to vertex $v$ and $c(v)=c_2.$ Since $\chi(G)\geq 3$ a vertex $w \in V(G)$ exists with $c(w)=c_j,$ $j\neq 1,2$ and edge $uw\notin E(G).$ Obviously the colour interchange $c(u)=c_j$ and $c(w)=c_1$ is possible and the colouring remains a Lucky colouring. Also, $\zeta(G)$ remains constant. Whereas before the colour interchange, edges $uw' \in E_\chi(G),$ $\forall~w'\in \{v':c(v')=c_j, v'\neq w\},$ these edges do not exist after the colour change. Hence, not all Lucky colourings yield identical vertex partitions. Thus, $G$ is not $SCC.$ 
\end{proof}
For any proper colouring of a graph with $k$ colours the vertex set can be partitioned into $k$ independent subsets say, $X_i,$ $i=1,2,3,\dots,k.$ It is obvious that $c(X_i)\neq c(X_j)$ if and only if $i\neq j.$ Call these vertex subsets, \emph{chromatic vertex subsets} of $V(G).$
\begin{theorem}
\label{thm2.10}
A graph $G$, $\zeta(G)>0$ is $SCC$ if and only if for any Lucky colouring the chromatic vertex subsets, $X_i$, $i=1,2,3,\dots,\chi(G)$ are such that every vertex $v\in X_i$, $\forall i$ is adjacent to at least one vertex $u\in X_j$, $j=1,2,3,\dots,i-1,i+1,\dots,\chi(G)$.
\end{theorem}
\begin{proof}
For a graph $G$ for which a Lucky colouring exists such that the chromatic vertex subsets, $X_i$, $i=1,2,3,\dots,\chi(G)$ are such that every vertex $v\in X_i$, $\forall i$ is adjacent to at least one vertex $u\in X_j$, $j=1,2,3,\dots,i-1,i+1,\dots,\chi(G)$ implies that identical vertex partitions are yielded for all Lucky colourings. Hence, $G$ is $SCC$.\\\\
Conversely, let $G$ be $SCC$. Consider a Lucky colouring and its corresponding chromatic vertex subsets, $X_i$, $i=1,2,3,\dots,\chi(G)$. If a vertex $v\in X_i$ exists such that $v$ is not adjacent to any vertex in say, $X_j$, $i\neq j$ then $v$ may be coloured $c(X_j)$ whilst $\zeta(G)$ remains unchanged. Now the chromatic vertex subsets changed to include $X_i\backslash v$ and $X_j\cup\{v\}$. It implies that $G$ is not $SCC$ which is a contradiction. Hence, $v$ must be adjacent to at least one vertex $v_j\in X_j$, $\forall j$, $j=1,2,3,\dots,i-1,i+1,\dots,\chi(G)$.
\end{proof} 
An important implication is that a graph $G$ which satisfies the conditions of Theorem~\ref{thm2.10} in respect of its chromatic vertex subsets, has a unique set of chromatic completion edges.

Recall that the $\emph{rainbow neighbourhood convention}$ prescribed that we colour the vertices of a graph $G$ in such a way that $\mathcal{C}_1=I_1,$ the maximal independent set in $G,$ $\mathcal{C}_2=I_2,$ the maximal independent set in $G_1=G-\mathcal{C}_1$ and proceed like this until all vertices are coloured \cite{3,4,5}. In \cite{3} the concept of a rainbow neighbourhood yielded by vertex $u$ in graph $G$ was introduced. It is important to note that the graph $G$ had to be coloured in accordance to the rainbow neighbourhood convention to determine the rainbow neighbourhood number, $r_\chi(G)$ of a graph $G.$ Hence, well-defined colouring conventions may serve to generalise the concept of the rainbow neighbourhood number i.e. the number of vertices which yield rainbow neighbourhoods in $G.$ The rainbow neighbourhood number associated with a Lucky colouring of a graph $G$ will be denoted by, $r_{\mathcal{L}}(G).$
\begin{theorem}
\label{thm2.11}
A graph $G$ is $SCC$ if and only if for a Lucky colouring of $G,$ $r_{\mathcal{L}}(G)=|V(G)|=\nu(G).$
\end{theorem}
\begin{proof}
If $r_{\mathcal{L}}(G)=|V(G)|=\nu(G)$ then every vertex is adjacent to at least one coloured vertex of each colour in a Lucky colouring of $G$. By Theorem~\ref{thm2.10}  it follows that $G$ is $SCC.$

Conversely, if $G$ is $SCC$ then by Theorem ~\ref{thm2.10} the result follows.
\end{proof}
Since a Lucky colouring is a relaxation of the Rainbow Neighbourhood Convention it follows that if a graph is $SCC$ in respect of a colouring in accordance with the Rainbow Neighbourhood Convention, it is $SCC$ in accordance with a Lucky colouring. In fact, it follows that in respect of such $G$ both colouring conventions are congruent colourings. See \cite{3,4,5,12} for further results in respect of rainbow neighbourhood numbers.

\section{Chromatic completion and stability in respect of $\mathcal{J}$-colouring}
\label{s3}
Thus far the notion of chromatic completion of a graph related strictly to chromatic colourings by the convention of Lucky colourings. This requirement can be relaxed to generalise over all chromatic colourings $\varphi_\chi(G).$ Clearly, we have by analogy that, $\zeta_\varphi(G)\leq \zeta(G).$ We further generalise to a recently introduced colouring convention called Johan colouring or $\mathcal{J}$-colouring [6,7,8,11]. The chromatic completion edge set denoted by $E_{\mathcal{J}}(G),$  will be investigated. Corresponding to a $\mathcal{J}$-colouring the cardinality of the chromatic completing edge set is denoted by, $\zeta_{\mathcal{J}}(G) = |E_{\mathcal{J}}(G)|.$ Recall the definition from [11].

\begin{definition} 
\label{def3.1}
A maximal proper colouring of a graph $G$ is a \textit{Johan colouring} of $G,$ denoted by $\mathcal{J}$-colouring, if and only if every vertex of $G$ yields a rainbow neighbourhood of $G$. The maximum number of colours in a $\mathcal{J}$-colouring is called the \textit{$\mathcal{J}$-chromatic number} of $G,$ denoted by $\mathcal{J}(G).$
\end{definition}

\begin{figure}[htbp]
\begin{center}
\scalebox{0.5}{
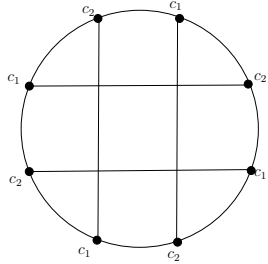}
\caption{$G$ with a Lucky colouring}
\label{e1}
\end{center}
\end{figure}

\begin{figure}[htbp]
\begin{center}
\scalebox{0.5}{
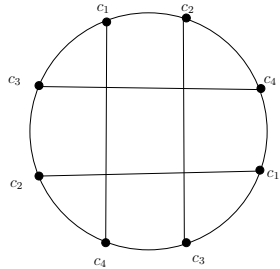}
\caption{$G$ with a $\mathcal{J}$-colouring}
\label{e2}
\end{center}
\end{figure}

Recall that not all graphs permit a $\mathcal{J}$-colouring [11]. Unless stated otherwise, all graphs in this subsection will permit a $\mathcal{J}$-colouring. Figure~\ref{e1}  depicts a graph with a Lucky colouring with $\chi(G)=2$ colours. It is easy to verify that $\zeta(G)=4.$ Figure~\ref{e2} depicts the same graph with a $\mathcal{J}$-colouring with $\mathcal{J}(G)=4$ colours. Note that the $\mathcal{J}$-colouring complies with the colour allocations of Lemma 2.6. It is easy to verify that, $ \zeta_{\mathcal{J}}(G)=|E_{\mathcal{J}}(G)|=12.$ The aforesaid serves as an illustration of the next result.
\begin{proposition}
\label{pro3.2}
For a graph $G$ it follows that, $\zeta_\varphi(G) \leq \zeta(G)\leq \zeta_{\mathcal{J}}(G).$
\end{proposition}
\begin{proof}
 That $\zeta_\varphi(G) \leq \zeta(G)$ holds follows directly from Theorem~\ref{thm2.5}. The result $\zeta(G)\leq \zeta_{\mathcal{J}}(G)$ is a direct consequence of the number theoretical result (or optimal near-completion $\ell$-partition) as the number of colours i.e. $\ell,$ increases whilst the order of $G$ remains constant.
\end{proof}
Proposition~\ref{pro2.12} can informally be understood by saying, that if for a particular Lucky colouring, some vertices in some colour classes are allocated new colours, then more edges  are permitted in terms of the definition of chromatic completion of a graph.
\begin{theorem}
\label{thm3.3}
For a graph $G$ it follows that, $\zeta_\varphi(G) = \zeta(G)= \zeta_{\mathcal{J}}(G)$ if and only if $\varphi_\chi(G) \cong \varphi_{\mathcal{L}}(G) \cong \varphi_{\mathcal{J}}(G).$
\end{theorem}
\begin{proof}
If $\varphi_\chi(G) \cong \varphi_{\mathcal{L}}(G) \cong \varphi_{\mathcal{J}}(G)$ it follows trivially that, $\zeta_\varphi(G) = \zeta(G)= \zeta_{\mathcal{J}}(G).$

If $\zeta_\varphi(G) = \zeta(G)= \zeta_{\mathcal{J}}(G)$ it follows from Theorem~\ref{thm2.10} that, $\varphi_\chi(G) \cong \varphi_{\mathcal{L}}(G) \cong \varphi_{\mathcal{J}}(G).$
\end{proof}
The next corollaries are direct consequences of Theorem~\ref{thm3.3}
\begin{corollary}
\label{col3.4}
For a graph $G$ equality holds in accordance with Theorem~\ref{thm2.13}  if and only if $G$ is $SCC$ in respect of a proper colouring.
\end{corollary}
\begin{corollary}
\label{col3.5}
$E_\varphi(G)~is~unique\Rightarrow E_{\mathcal{L}}(G)~is~unique\Rightarrow E_{\mathcal{J}}(G)~is~unique.$
\end{corollary}
An example is that, the alternating colouring $c(v_i)=c_1,~c(v_{i+1})=c_2,$ $i=1,2,3,\dots, n-1$ of the vertices of an even cycle graph, $C_n,$ $n\geq 4$ is firstly, a $\mathcal{J}$-colouring because it is a maximal proper colouring with each vertex yielding a rainbow neighbourhood in $C_n.$ It follows easily that the colouring is indeed a Lucky colouring followed by the implication that it is a chromatic colouring. Clearly, for all three colouring conventions the even cycle graph is $SCC$ as well as, $ E_\varphi(G) = E_{\mathcal{L}}(G)= E_{\mathcal{J}}(G).$ and unique to $C_n,$ $n\geq 4,$ $n$ even.

Since any graph $G$ is $k$-colourable for some $k\geq 1$ it is chromatic colourable. Therefore, it permits a Lucky colouring. However not all graphs permit a $\mathcal{J}$-colouring. This observation leads to the next result.
\begin{theorem}
\label{thm3.6}
A graph $G$ which does not permit a $\mathcal{J}$-coloring is not $SCC.$
\end{theorem}
\begin{proof}
A graph $G$ which does not permit a $\mathcal{J}$-colouring has for all proper colourings, including all Lucky colourings, a corresponding colour class vertex partition such that a vertex $v,$ $c(v)=c_i$ exists which is not adjacent to at least one vertex in each colour class $\mathcal{C}_j$, $j=1,2,3,\dots,i-1,i+1,\dots,\chi(G).$ Hence, by Theorem~\ref{thm2.10}, $G$ is not $SCC.$
\end{proof}
Put differently, Theorem~\ref{thm3.6} implies that the set of $SCC$ graphs is a subset of the set of $\mathcal{J}$-colourable graphs. As example, it follows that any graph $G$ which contains an odd cycle of length $n\geq 5$ and $n\equiv1~or~2~(mod~3)$ is not $SCC.$ See [??] for $\mathcal{J}$-colourable results on odd cycle graphs.
\section{Elementary Graph Operations}
\label{s4}
In this subsection the disjoint union and the join of graphs $G$ and $H$ are considered. In the disjoint union operation between graphs $G$ and $H$ the respective values, $\zeta(G)$ and $\zeta(H)$ remain the same if $\varphi_{\mathcal{L}}(G),$, $\varphi_{\mathcal{L}}(H)$ remain the same. The term $\sum\limits_{i=1}^{\chi(G)}\prod\limits_{j=1}^{\chi(H)}\theta_G(c_i)\theta_H(c_j)_{i\neq j}$ follows from the definition of chromatic completion of a graph. 
\begin{proposition}
\label{pro4.1}
For graphs $G$ and $H$ it follows that, $\zeta(G\cup H) \geq \zeta(G)+\zeta(H) + \sum\limits_{i=1}^{\chi(G)}\prod\limits_{j=1}^{\chi(H)}\theta_G(c_i)\theta_H(c_j)_{i\neq j}.$  
\end{proposition}
\begin{proof}
For $\varphi_{\mathcal{L}}(G)$ and $\varphi_{\mathcal{L}}(H)$ it is possible (not necessarily) to find a pair of vertices $u,v \in V(G)$ such that $c(u)=c(v)$ in $G$ but $c(u)\neq c(v)$ in $G\cup H.$ Therefore, the edge $uv \in E_\chi(G\cup H)$ but $uv\notin E_\chi(G).$ Similarly in $H.$ It is also possible (not necessarily) to find a pair of vertices $u\in V(G),$ $v\in V(H)$ such that $c(u)=c(v)$ in $G$ and $H$ respectively, but $c(u)\neq c(v)$ in $G\cup H.$ Therefore, the edge $uv \in E_\chi(G\cup H).$ Hence, $\zeta(G\cup H) \geq \zeta(G)+\zeta(H) + \sum\limits_{i=1}^{\chi(G)}\prod\limits_{j=1}^{\chi(H)}\theta_G(c_i)\theta_H(c_j)_{i\neq j}.$
\end{proof}
The complexity to improve on the lower bound of Proposition 3.1 stem from the facts that firstly, $\chi(G\cup H) = max\{\chi(G), \chi(H)\}$ and secondly, that the value $\chi(G\cup H)$ must be applied to $\nu(G\cup H)=\nu(G)+\nu(H)$ vertices to find an appropriate optimal near-completion $\ell$-partition.

\begin{conjecture} 
\label{con4.2}
 If both graphs $G,$ $H$ permit a Lucky colouring as prescribed in accordance to  Theorem~\ref{thm2.5} and Lemma~\ref{lem2.6} then, $\zeta(G\cup H) = \zeta(G)+\zeta(H) + \sum\limits_{i=1}^{\chi(G)}\prod\limits_{j=1}^{\chi(H)}\theta_G(c_i)\theta_H(c_j)_{i\neq j}.$ 
\end{conjecture}

The graph $(G-E(G))+(H-E(H))$ is a spanning subgraph of $G+H$ and chromatic completion does not result in additional edges $uv$, $u\in V(G)$, $v\in V(H).$ Since $\chi(G+H)=\chi(G)+\chi(H)$ with colours say, $\mathcal{C}=\{c_1,c_2,c_3,\dots,c_{\chi(G)}, c_{\chi(G)+1}, c_{\chi(G)+2}, c_{\chi(G)+3},\dots, c_{\chi(G)+\chi(H)}\},$ the values $\zeta(G)$ and $\zeta(H)$ remain the same in the join operation between graphs $G$ and $H.$ Hence, $\zeta(G+H) = \zeta(G)+\zeta(H)$ for all graphs. It thus follows that $\zeta(K_1+ H)= \zeta(H)$. Despite this trivial observation it is hard to find $\zeta(G\circ H)$ in general, where $G\circ H$ denote the corona graph.
\begin{corollary}
\label{col4.3}
$G$ and $H$ are $SCC$ if and only if $G + H$ is $SCC.$
\end{corollary}
\begin{proof}
If $G+H$ is $SCC$ then a vertex $v\in V(G)$, $c(v)=c_i$ is adjacent to at least one vertex $u$, $c(u)=c_j,$ $j= 1,2,3,\dots i-1,i+1,\dots, \chi(G)+\chi(H).$ Therefore, $v\in V(G),$ $c(v)=c_i$ is adjacent to at least one vertex $u \in V(G+H),$ $c(u)=c_j,$ $j= 1,2,3,\dots i-1,i+1,\dots, \chi(G)+\chi(H).$ Hence, $v\in V(G),$ $c(v)=c_i$ is adjacent to at least one vertex $u \in V(G),$ $c(u)=c_j,$ $j= 1,2,3,\dots i-1,i+1,\dots, \chi(G).$ Thus, by Theorem~\ref{thm2.10}, $G$ is $SCC.$ Similarly it follows that $H$ is $SCC.$

Conversely, if both $G$ and $H$ are $SCC$ the join operation implies that $uv\in E(G+H),$ $\forall$ distinct pairs $u\in V(G),$ $v\in V(H).$ Hence, $G+H$ is $SCC.$ Clearly, if say $G$ is not $SCC$ and $H$ is $SCC$ then, by Theorem~\ref{thm2.10}, $G+H$ is not $SCC$ because there exists at least one vertex $v\in V(G),$ $c(v)=c_i$ which is not adjacent to at least one vertex coloured $c_j,$ $1,2,3,\dots i-1,i+1,\dots, \chi(G).$
\end{proof}
Proposition~\ref{pro4.1} read together with $\zeta(G+H) = \zeta(G)+\zeta(H)$ implies that, $\zeta(G\cup H) \geq \zeta(G+H).$
\section{Conclusion}
Lucky's theorem read with Lemma 2.2 allows for the determination of an upper bound of $\zeta(G)$. Note that for two proper colorings say, $\varphi:V(G)\mapsto \mathcal{C}$, $|\mathcal{C}|=k$ and $\varphi':V(G)\mapsto \mathcal{C}'$, $|\mathcal{C}'| =k+t$, $1\leq t\leq n-k$, which are both allocated according to Lemma 2.6 or as a near-completion $k$-partition and a near-completion $(k+t)$-partion respectively, then $\varepsilon(G_{{\varphi}'}) \geq \varepsilon(G_\varphi)$, where $\varepsilon(G_{{\varphi}'})$, $\varepsilon(G_\varphi)$ denote the respective proper colouring completion graphs. The aforesaid leads to the following algorithm which provides an upper bound on $\zeta(G)$. For the purpose of the algorithm we utilise a deviation of the minimum parameter indexing of the vertices.
\subsection{Near-Lucky proper coloring of graph $G$}
Consider a graph $G$ of order $n\geq 1$ with vertices $v_i$, $i=0, 1,2,\dots, n-1$ such that, $d(v_0)\geq d(v_1)\geq d(v_2)\geq\cdots\geq d(v_{(n-1)})$. Let $\varphi:V(G)\mapsto \mathcal{C}$, $\mathcal{C} =\{c_0,c_1,c_2,\dots,c_{(\chi(G)-1)}\}$ be a chromatic colouring of $G$.\\\\
Step 1: Let $j=1$, $\mathcal{C}_j = \mathcal{C}$ and $i=j$. Go to step 2.\\
Step 2: Let $\mathcal{C}_i = \mathcal{C}_j$. Colour $v_i$, $c(v_i)=c_i~(mod~\chi(G))$ if possible, alternatively any permissible colour $c_t \in \mathcal{C}_i$ if possible. Else, colour with and additional colour $c_i'\notin \mathcal{C}_i$ and let $\mathcal{C}_j= \mathcal{C}_i\cup \{c'_i\}$. Go to step 3.\\
Step 3: Let $j=i+1$. If $j>n-1$, go to step 4. Else, let $i=j$. Go to step 2.\\
Step 4:  Let $\varphi':V(G)\mapsto \mathcal{C}_i$ and stop.\\\\
Clearly the algorithm is finite and results in $|\mathcal{C}_n|\geq |\mathcal{C}|$. Hence, it follows easily from Lucky's theorem read with Lemma 2.2 that, $\zeta(G)=\zeta_{\varphi}(G) \leq \zeta_{\varphi'}(G)$. It is easy to verify that the algorithm results in an exact Lucky colouring for all cycle graphs, $C_n$, $n\geq 3$ if the verices are consecutively labeled. Hence, the upper bound is best possible if a Lucky colouring results. Improving the efficiency of the algorithm remains open. The authors suggest that for graphs in general, coloring $v_i$ followed by colouring $N(v_i)$, $i\in \{0,1,2,\dots, n-1\}$ until all vertices have been coloured is a worthy avenue of research.

Improving the upper bound in Theorem 2.4 for graphs which are not complete remains open.

Theorem~\ref{thm2.10} characterizes graphs which are $SCC.$ Finding well-defined families of graphs which are $SCC$ remains open and is certainly worthy of further research. Furthermore, if a graph $G$ is $SCC$ it implies that for all Lucky colourings the corresponding chromatic completions graphs are isomorphic. However, it is possible to find graphs which are not $SCC$ but up to isomorphism, meaning that the vertices may require relabeling, the chromatic completion graphs remain equivalent. Finding such graphs remains open for research.

The difficulty seen with the disjoint union signals that advanced graph operations such as graphs products, graph powers, derivative corona's and others permit worthy research. In fact finding a Lucky colouring for a given graph $G$ in general remains an open problem. Characterizing graphs $G,$ $H$ such that $G\circ H$ is $SCC$ and doing same for other graph operations remain open problems.

\end{document}